\documentclass[a4, 11pt]{amsart}
\usepackage{amssymb,latexsym}
\usepackage{color}
\usepackage{graphicx}
\usepackage[left=25mm, right=25mm, top=30mm, bottom=30mm, includefoot, includehead]{geometry}
\theoremstyle{plain}
\newtheorem{theorem}{Theorem}[section]
\newtheorem{corollary}[theorem]{Corollary} 
\newtheorem{proposition}[theorem]{Proposition}
\newtheorem{lemma}[theorem]{Lemma}
\newtheorem{conjecture}[theorem]{Conjecture}
\theoremstyle{definition}
\newtheorem{definition}[theorem]{Definition}
\newtheorem{example}[theorem]{Example}

\newtheorem{remark}[theorem]{Remark}
\newtheorem{question}[theorem]{Question}

\newcommand{\enm}[1]{\ensuremath{#1}}          %
\newcommand{\CC}{\enm{\mathbb{C}}}

\newcommand{\NN}{\enm{\mathbb{N}}}

\newcommand{\PP}{\enm{\mathbb{P}}}

\renewcommand{\phi}{\varphi}
\renewcommand{\theta}{\vartheta}
\renewcommand{\epsilon}{\varepsilon}

\newcommand{\old}[1]{}

\makeatletter
\@namedef{subjclassname@2020}{%
	\textup{2020} Mathematics Subject Classification}
\makeatother

\date{}


\title{On the Koiran-Skomra's question about Hessians}

\author{Edoardo Ballico}
\address{Universit\`{a} di Trento, Via Sommarive 14,  38123 Povo (Trento), Italy}
\email{edoardo.ballico@unitn.it}

\author{Emanuele Ventura}
\address{Politecnico di Torino, Dipartimento di Scienze Matematiche ``G. L. Lagrange'', Corso Duca degli Abruzzi 24\\
10129 Torino, Italy}
\email{emanuele.ventura@polito.it}

\begin{document}

\maketitle

\begin{abstract}
We give a negative answer to a question of Koiran and Skomra about Hessians, motivated by Kayal's algorithm for the equivalence problem to the Fermat polynomial. We conjecture that our counterexamples are the only ones. We also study a local version of their question. 
\end{abstract}

\section{Introduction}

\noindent Let $g\in \mathbb C[x_1,\ldots, x_n]_d$ be a homogeneous polynomial of degree $d$ and let $\mathrm{Hess}(g)$ be its Hessian matrix. The {\it Hessian map} is the polynomial map 
\[
H: \CC[x_1,\ldots, x_n]_d\longrightarrow \CC[x_1,\ldots, x_n]_{n(d-2)},
\] 
defined by $H(g) = \det(\mathrm{Hess}(g))$. 

Kayal's algorithm takes as input a homogeneous polynomial $g\in \mathbb C[x_1,\ldots,x_n]_d$ and determines whether $g$ is in the $\mathrm{GL}(\mathbb C^n)$-orbit of (or {\it it is equivalent to}) the Fermat polynomial $f = x_1^d + \cdots + x_n^d$; if it is so, the algorithm outputs linearly independent linear forms $\ell_i\in \CC[x_1,\ldots,x_n]_1$ such that $g = \sum_{i=1}^n \ell_i^d$. See \cite{Kayal} and \cite[\S 3.1]{KS} for a detailed account. 
This algorithm is based on three steps: 
\begin{enumerate}
\item[(1)] Check that the Hessian determinant $H(g)$ is nonzero and can be factored
as $H(g) = \alpha \prod_{i=1}^n h_i^{d-2}$ where $h_i\in \mathbb C[x_1,\ldots,x_n]_1$ are linear forms. 
If this is not possible, reject. 

\item[(2)] Try to find complex constants $\alpha_i\in \mathbb C$ such that 
\[
g = \sum_{i=1}^n \alpha_i\ell_i^d. 
\]
If this is not possible, reject. 

\item[(3)] Declare $g$ to be equivalent to $f$ and output $\ell_i = \beta_i h_i$ where $\beta_i^d = \alpha_i$. 
\end{enumerate}

\noindent Motivated by Kayal's algorithm, Koiran and Skomra \cite[Question 1]{KS} asked the following question: 

\begin{question}[Koiran-Skomra]\label{quest}
Let $n\geq 2$ and $d\geq 3$. Let $f\in \mathbb C[x_1,\ldots, x_n]_d$ be such that $H(f) = \alpha (x_1\cdots x_n)^{d-2}$, for some $\alpha\neq 0$. 
\begin{center}
\emph{Does it follow that $f = \alpha_1x_1^d + \cdots + \alpha_n x_n^d$, for some 
constants $\alpha_j\in \mathbb C$?}
\end{center}
\end{question} 
This question is interesting also because it is so simple to state and has a delightful invariant theory flavour. Hessians of forms have been the subject of many classical works in algebraic 
geometry and commutative algebra (for instance, Hessians are related to Lefschetz properties of artinian Gorenstein algebras); see e.g. \cite[Chapter 7]{russo} and \cite{CRS}, along with the several references therein. Hessians are also naturally related to scheme-theoretic ranks of homogeneous polynomials \cite{DG,HMV}. Despite of this large body of results, we are not aware of any classical algebraic or geometric source taking this natural question into consideration. 

Employing the algorithm in \cite[\S 4.1]{KR} to decide whether the factorization in step (1) of Kayal's algorithm {\it exists} (and not to compute the explicit factorization), a positive answer to Question \ref{quest} would provide a polynomial time algorithm for the equivalence problem over $\mathbb C$. 
In their article \cite{KS}, Koiran and Skomra did provide a polynomial time algorithm for the equivalence problem over $\mathbb C$ in degree $d=3$, which was later extended by Koiran and Saha \cite{KSa} to $d>3$. However, their algorithm has to work in a very different way than Kayal's algorithm because of our Theorem \ref{negativeresult}: this shows that over any field $K\subset \CC$, the answer is negative for infinite pairs of integers $(n,d)$. An instance of interest in \cite{KS} is $d=3$: in this case, the answer is positive if and only if $n=2$ (Corollary \ref{cased=3}). 

Our monomial counterexamples are {\it homaloidal polynomials}, i.e. their first partial derivatives define a Cremona transformation. Theorem \ref{negativeresult} also shows that these monomials are {\it totally Hessian} \cite[Remark 3.5]{CRS}. Therefore 
they give examples of totally Hessian polynomials (although reducible) of arbitrarily large degree in any number of variables $n\geq 2$. 

In Theorem \ref{binaryodd} and Theorem \ref{binaryeven}, we deal with the case of singular binary homogeneous polynomials, providing the only ones whose Hessians have the desired form. Our approach is combinatorial
and we wonder whether an invariant theoretic strategy could be put in place. We conjecture that for any $n\geq 2$ the only smooth homogeneous polynomials satisfying the assumptions of Question \ref{quest} are the Fermat polynomials; see Conjecture
\ref{onlyfermat}.  

Computing the differential of the Hessian map, we formulate a local version of Koiran-Skomra's question, 
which has a positive answer when $d\geq n+1$ (Proposition \ref{prop:dim of ker fermat}). \\

\noindent {\bf Acknowledgements.}
We thank Enrico Carlini, Pascal Koiran, Francesco Russo, and Mateusz Skomra for useful discussions. The second author would like to thank Tony Iarrobino, Pedro Marques, Maria Evelina Rossi, and Jean Vall\`{e}s for organizing the nice workshop {\it Deformation of Artin algebras and Jordan type} within the AMS-EMS-SMF meeting in Grenoble on 18-22 July 2022. 

\section{Koiran-Skomra's question}
\subsection{Negative result} 

The counterexamples to Question \ref{quest} that we find are direct sums of homogeneous polynomials
of the form prescribed by the next result. 

\begin{theorem}\label{negativeresult}
Let $K\subset \mathbb C$ be a field. The answer to Question \ref{quest} over $K$ is negative for all $n\geq 2$ and $d\geq 3$ such that there exists $2\leq q\leq n$ with $q \mid d$. 
\begin{proof}
Let $d = qk$ with $k\geq 1$ (or $k\geq 2$ if $q=2$). Then define
\[
g = (x_1\cdots x_q)^{k} + x_{q+1}^d + \cdots + x_n^d\in K[x_1,\ldots, x_n].
\]
We show that 
\begin{equation}\label{firstidentityhessian}
H((x_1\cdots x_q)^{k}) = (1-d)(-k)^q(x_1\cdots x_q)^{d-2}\neq 0. 
\end{equation}
This is enough to conclude that $H(g)$ has the desired form. We first prove that every monomial in the determinant expansion of the Hessian above is $(x_1\cdots x_q)^{d-2}$. 

 To see this, fix an index $\ell\in \lbrace 1,\ldots, q\rbrace$ and look at the total degree in $x_{\ell}$ of an arbitrary monomial 
in the determinant expansion. By symmetry, we may assume $x_{\ell} = x_1$. 

Let $f_{i,j}= \left(\mathrm{Hess}((x_1\cdots x_q)^{k})\right)_{i,j}$ be the $(i,j)$-entry of the Hessian matrix of $(x_1\cdots x_q)^{k}$. Let $m$ be a monomial in the expansion of $H((x_1\cdots x_q)^{k})$. This is a product of entries $f_{i,j}$, where the product ranges over the indices of all rows and columns, by definition of determinant. 

Then exactly one of the following cases can occur: 
\begin{enumerate}
\item[(I)] $m$ is a multiple of the $(1,1)$-entry (this choice gives no contribution for $k=1$); 
\item[(II)] $m$ is a multiple of two distinct entries in the first row and in the first column. 
\end{enumerate}
In case (I), the total degree of $m$ with respect to $x_1$ is $(k-2) + (q-1)k = qk-2=d-2$. In case (II), the total degree of $m$ with respect to $x_1$ is $2(k-1)+(q-2)k=qk-2=d-2$. In other words, for any monomial in the expansion of $H((x_1\cdots x_q)^{k})$, the total degree of every variable is $d-2$. 

Let $C$ be the $q\times q$ matrix whose $(i,j)$-entry is the coefficient of the monomial $f_{i,j}$. One has
\begin{equation}\label{matrixofdegformon}
C =  \begin{bmatrix}
k(k-1) & k^2 & \cdots & k^2 \\
k^2 & k(k-1) &  \cdots & k^2 \\
\vdots & \vdots & \vdots & \vdots \\
k^2 & k^2 & \cdots & k(k-1) \\
\end{bmatrix}. 
\end{equation}
The argument above shows that 
\[
H((x_1\cdots x_q)^{k}) = \det(C)(x_1\cdots x_q)^{d-2}. 
\]
Write $C = -kI_{q} + k^2 J_q$, where $I_{q}$ is the $q\times q$ identity matrix and $J_q$ is the $q\times q$ matrix whose entries are all ones. Note that $J_q = {\bf 1}_q \cdot ({\bf 1}_q)^{T}$, where ${\bf 1}_q$ is the column vector with $q$ ones. 

\noindent One may write $C =  -kI_{q} \cdot C'$, where $C' = I_{q} - {\bf 1}_q\cdot (k{\bf 1}_q)^T$. Note that 
\begin{equation}\label{matrices}
\begin{bmatrix}
I_q &  0 & \\
(k{\bf 1}_q)^T & 1 \\
\end{bmatrix} 
\cdot 
\begin{bmatrix}
I_q  - ({\bf 1}_q)\cdot (k{\bf 1_q})^T&  -{\bf 1}_q & \\
0 & 1 \\
\end{bmatrix} 
\cdot 
\begin{bmatrix}
I_q  &  0 & \\
-(k{\bf 1}_q)^T & 1 \\
\end{bmatrix} 
= 
\begin{bmatrix}
I_q &  -{\bf 1}_q & \\
0 & 1 - (k{\bf 1}_q)^{T}\cdot {\bf 1}_q\\
\end{bmatrix}. 
\end{equation}
Then $\det(C')$ is the determinant of the left-hand side of \eqref{matrices}, thus the determinant of the right-hand side too. So $\det(C') = 1- (k{\bf 1}_q)^T\cdot {\bf 1}_q = 1-qk$ and $\det(C) = (1-d)(-k)^q\neq 0$, which shows \eqref{firstidentityhessian}. 

Finally, the polynomial $g$ has Hessian $H(g) = \alpha (x_1\cdots x_n)^{d-2}\neq 0$ and is not a smooth homogeneous polynomial over $K$. Hence $g$ cannot be equivalent to a Fermat polynomial over $K$.  
\end{proof}
\end{theorem}

\begin{remark}
A homogeneous polynomial $f\in \CC[x_1,\ldots,x_n]$ is {\it homaloidal} when its first partial derivatives define a birational gradient map $\mathrm{grad}(f): \PP^n \dashrightarrow \PP^n$ (i.e. a Cremona transformation). Equivalently, their {\it polar degree} is one. 

Our monomial counterexamples are homaloidal polynomials: their reduced polynomial 
$x_1\cdots x_n$ is homaloidal and the very deep \cite[Corollary 2]{DP} shows that the polar 
degree of a polynomial coincides with the one of the corresponding reduced polynomial. Theorem \ref{negativeresult} also shows that these polynomials are {\it totally Hessian} \cite[Remark 3.5]{CRS} (although reducible) of arbitrarily large degree in any number of variables $n\geq 2$. 
\end{remark}

\begin{remark}
The homogeneous polynomial $g$ in the proof of the above theorem is a sum of pairwise coprime monomials. Then \cite[Theorem 3.2]{CCG} shows that the complex Waring rank of $g$ satisfies the equality
\[
\mathrm{rk}_{K}(g)\geq  \mathrm{rk}_{\mathbb C}(g) =  (k+1)^{q-1} + (n-q) > n = \mathrm{rk}_{K}(f),
\]
where $f=x_1^d+\cdots + x_n^d$. 
\end{remark}

\begin{remark}\label{exampleothermons}
There exist homogeneous polynomials whose Hessians are monomials distinct from $(x_1\cdots x_n)^{d-2}$. Three examples: 
\begin{enumerate}
\item[(i)] for $k\geq 2$, let $f = x_2^{k-1}x_1^2 - x_3^{k+1}\in \CC[x_1,x_2,x_3]_{k+1}$. It has singularities at $(0:1:0)$ and $(1:0:0)$, and $H(f) = -2k^2(k+1)(k-1) x_1^2 x_2^{2(k-2)}x_3^{k-1}$. 

\item[(ii)] $g = x_1(x_2^2+x_1x_3 + \sum_{i=4}^{n} x_i^2)\in \CC[x_1,x_2,x_3,\ldots, x_n]_{3}$ is a singular cubic form
which splits as a product between a smooth quadric and a tangent hyperplane to it. Then $H(f) = -2^n x_1^n$. 

\item[(iii)] The determinant of a generic sub-Hankel matrix \cite[4.1.1]{CRS} is a homaloidal polynomial whose Hessian is a power of a variable \cite[Theorem 4.4]{CRS}. Since these are homaloidal of degree $d\geq 3$, they are singular. 

\end{enumerate}
\end{remark}

A case of interest in \cite{KS} is when $d=3$. 

\begin{corollary}\label{cased=3}
Let $d=3$. Then the answer to Question \ref{quest} is positive if and only if $n=2$. 
\begin{proof}
If $n\geq 3=d$, by Theorem \ref{negativeresult} the answer is negative. Suppose $n=2$ and let
$f = a_3x_1^3 + a_2x_1^2x_2 + a_1x_1x_2^2 + a_0x_2^3$. So
\[
H(f) = (-4a_{2}^2+12a_{3}a_{1})x_1^2+(-4a_{2}a_{1}+36a_{3}a_{0})x_1x_2+(-
4a_{1}^2+12a_{2}a_{0})x_2^2, 
\]
\noindent where $H(f) = \alpha x_1x_2\neq 0$. If $a_{2}=0$, then $a_{3}\neq 0$ and so 
$a_{1}=0$. But then $f = a_3 x_1^3 + a_0x_2^3$. If not, by symmetry, we may assume $a_{1}, a_{2}\neq 0$. Therefore $a_{0} = a_{1}^2/3a_{2}$ and $a_{3} = a_{2}^2/3a_{1}$, which implies that the coefficient of the monomial $x_1x_2$ in $H(f)$ is zero. 
\end{proof}
\end{corollary}

For $n=d=3$, the following statement is based on a well-known classification and we record it from our perspective. 

\begin{proposition}\label{d=3}
For $n=d=3$, the only cubic homogenous polynomials satisfying the assumptions in Question \ref{quest} are $f = a_1x_1^3 + a_2x_2^3 + a_3x_3^3$ and $f = a_{123} x_1x_2x_3$ (with $a_1a_2a_3\neq 0$ and $a_{123}\neq 0$ resp.). 
\begin{proof}
Let $f\in \mathbb C[x_1,x_2,x_3]$ be a cubic form. The Hessian map
$H: \mathbb C[x_1,x_2,x_3]_{3}\rightarrow \mathbb C[x_1,x_2,x_3]_{3}$,
$H: g\mapsto H(g)$ is an $\mathrm{SL}(\mathbb C^3)$-coinvariant \cite[\S 5.1]{Dolg}. Thus, it is enough to check 
the $\mathrm{SL}(\mathbb C^3)$-orbits of ternary cubic homogenous polynomials, which are well-known. Among these, we see that the only cubic homogenous polynomials whose Hessian splits into three linearly independent linear forms are the Fermat polynomials 
$g = \ell_1^3 + \ell_2^3 + \ell_3^3$ and the triangles $g=\ell_1\ell_2\ell_3$. 
Since by assumption $H(g) = \alpha x_1x_2x_3\neq 0$, by Lemma \ref{autofhess} proved below, we see that 
$\ell_i = x_i$ (up to relabeling and scaling). This shows the statement. 
\end{proof}
\end{proposition}

\subsection{Singular binary homogeneous polynomials}

We deal with the case of singular binary homogeneous polynomials, providing the only ones whose Hessians have the desired form.

\begin{remark}\label{factsonhessians}
For $n\geq 2$, let $f\in \CC[x_1,\ldots, x_n]_d$ and let $X_f\subset \PP^{n-1}$ be the reduced projective hypersurface defined by $f$. Then the reduced singular locus $\mathrm{Sing}(X_f)$ sits inside $X_{H(f)}$, where $H(f) = \det(\mathrm{Hess}(f))$
\cite[Proposition 1.1.19]{Dolgachev}. 
\end{remark}

\begin{theorem}\label{binaryodd}
For $n=2$ and $d\geq 3$ odd, there are no singular homogeneous polynomials satisfying the assumptions in Question \ref{quest}. 
\begin{proof}
Let $d = 2p+1$. Let $f = f(x_1,x_2) = \sum_{i=0}^d a_i x_1^i x_2^{d-i}\in \CC[x_1,x_2]_d$ be such that $H(f) = \alpha (x_1x_2)^{d-2}\neq 0$. Up to the automorphism of $\CC[x_1,x_2]$ swapping $x_1$ and $x_2$ and by Remark \ref{factsonhessians}, we may assume that $f = \sum_{i\geq 2}^d a_i x_1^i x_2^{d-i}$, i.e. $a_0 = a_1 = 0$. The coefficient of the monomial $x_1^{j}x_2^{2d-j-4}$ in $H(f)$ is of the form
\[
\sum_{k+\ell=j+2} \gamma_{(k,\ell)} a_k a_{\ell}. 
\]
For all $2\leq s\leq p$ and $j=2(s-1)$, the coefficient $\gamma_{(s,s)}$ in front of $a_{s}^2$ appearing in the monomial
$x_1^{2s-2}x_2^{2d-2s-2}$ in $H(f)$ is 
\[
\gamma_{(s,s)} = -s(d-1)(d-s)\neq 0. 
\]

Now, suppose that for some $1\leq j'<p$ we have that $a_{j''}=0$ for all $0\leq j''\leq j'$. We claim that $a_{j'+1}=0$. 

To show the claim, look at the coefficient of the monomial $x_1^{2j'}x_2^{2d-2j'-4}$ given by
\begin{equation}\label{coeff2j'+2}
\sum_{k+\ell=2j'+2} \gamma_{(k,\ell)} a_k a_{\ell}. 
\end{equation}
Since $k+\ell = 2(j'+1)$, if $\ell \geq j'+2$ then $k\leq j'$. By assumption, in this case, we have 
$a_ka_{\ell}=0$ and so this summand does not contribute to \eqref{coeff2j'+2}. 
Then the only summand left in \eqref{coeff2j'+2} is $\gamma_{(j'+1,j'+1)} a_{j'+1}^2$. 
Since $j'<p$, $2j'\leq 2(p-1)= 2p-2 = d-3$. Thus the monomial $x_1^{2j'}x_2^{2d-2j'-2}$ does not appear in $H(f)$, by assumption. Therefore $\gamma_{(j'+1,j'+1)} a_{j'+1}^2 = 0$ and since $\gamma_{(j'+1,j'+1)}\neq 0$, we find $a_{j'+1}=0$. 

Since $a_0=a_1=0$, then the argument above proves that $a_j=0$ for all $j\leq p$. Now, look
at the coefficient of $x^{d-2}y^{d-2}$. This is of the form 
\[
\sum_{k+\ell=d} \gamma_{(k,\ell)} a_k a_{\ell}. 
\]
If $\ell\geq p+1$, then $k\leq p$. So $a_{k}a_{\ell} = 0$ by what we have shown. Hence the coefficient above vanishes. This is in contradiction with the assumption on $H(f)$. 
\end{proof}
\end{theorem}

\begin{theorem}\label{binaryeven}
For $n=2$ and $d\geq 4$ even, the only singular binary homogeneous polynomials satisfying the assumptions in Question \ref{quest} are of the form $f = a_{d/2}x_1^{d/2}x_2^{d/2}$. 
\begin{proof}
Let $d = 2p$. Let $f = f(x_1,x_2) = \sum_{i=0}^d a_i x_1^i x_2^{d-i}\in \CC[x_1,x_2]_d$ be such that $H(f) = \alpha (x_1x_2)^{d-2}\neq 0$. As in the proof of Theorem \ref{binaryodd}, we may assume that $f = \sum_{i\geq 2}^d a_i x_1^i x_2^{d-i}$, i.e. $a_0 = a_1 = 0$. Moreover, again looking at $\gamma_{j,j}$ for $j\leq p-1$, we find that $a_{j}=0$ for all $j\leq p-1$ and $a_{p}= a_{d/2}\neq 0$. 

Let $1\leq t\leq p$ and consider the coefficient $a_{p+t}$ of the monomial $x_1^{p+t}x_2^{d-p-t}$ in $f$. 
Look at the coefficient of the monomial $x_1^{2p+t-2}x_2^{2p-t-2}$ in $H(f)$. This is of the form
\begin{equation}\label{coeff2p+t-2}
\sum_{k+\ell=2p+t} \gamma_{(k,\ell)} a_k a_{\ell}
\end{equation}
and, since $2p+t-2>2p-2=d-2$, it must vanish by the assumption on $H(f)$. 
Since $k+\ell=2p+t$, if $\ell\geq p+t+1$ then $k\leq p-1$. In this case, $a_k a_{\ell}=0$ by what we have shown above. 
So the only left summand in \eqref{coeff2p+t-2} is $\gamma_{(p,p+t)}a_pa_{p+t}$, where $a_p\neq 0$ but
$\gamma_{(p,p+t)}a_pa_{p+t}=0$. 

If $\gamma_{(p,p+t)}\neq 0$, for any given $p$ and every $1\leq t\leq p$, then $a_{p+t}=0$ for all $1\leq t\leq p$ 
and we would be done. 

The coefficient $\gamma_{(p,p+t)}$ is given by
\begin{equation}\label{coefft^2}
\gamma_{(p,p+t)} = 2p(2p-1)(t^2-p). 
\end{equation}
In other words, given $p\geq 2$, $\gamma_{(p,p+t)} = 0$ if and only if $t=\sqrt{p}$ (i.e. when $p$ is a perfect 
square and $t$ is its positive square root). 

From now on, we shall then assume that $p$ is a perfect square, otherwise we are done. Since $\gamma_{(p,p+t')}\neq 0$ for all $1\leq t'\leq \sqrt{p}-1$, then the coefficient $\eqref{coeff2p+t-2}$ for the monomial $x_1^{2p+t'-2}x_2^{2p-t'-2}$ is 
\[
\gamma_{(p,p+t')}a_pa_{p+t'}=0. 
\]
Hence $a_{p+t'}=0$ for all $1\leq t'\leq \sqrt{p}-1$. 

{\bf Claim 1}:
Let $1\leq t\leq p$ be an integer such that $m\sqrt{p}  < t < (m+1)\sqrt{p}$, for some integer $0\leq m\leq \sqrt{p}-1$. 
Suppose that $a_{p+t'}=0$ for all $t'$ such that $m'\sqrt{p} < t'  < (m'+1)\sqrt{p}$ for all $0\leq m' < m$. 
Then $a_{p+t}=0$. 
\begin{proof}
Write $t = m\sqrt{p} + n$, where $1\leq n < \sqrt{p}$. 
We look at the coefficient of the monomial $x_1^{2p+t-2}x_2^{2p-t-2} = x_1^{2p+m\sqrt{p}+n-2}x_2^{2p-m\sqrt{p}-n-2}$. 
This is of the form 
\begin{equation}\label{coeff2p+msqrtp+n-2}
\sum_{k+\ell=2p+m\sqrt{p}+n} \gamma_{(k,\ell)} a_k a_{\ell}.
\end{equation}
Note that $2p+m\sqrt{p}+n-2 > 2p-2=d-2$ and so the coefficient \eqref{coeff2p+msqrtp+n-2} must vanish. 
If $\ell$ or $k$ are of the form $p+q\sqrt{p}+q'$ with $0\leq q\leq m-1$ and $1\leq q'< \sqrt{p}$, then 
the product $a_k a_{\ell}=0$, by assumption. 
If $\ell = p + z\sqrt{p}$ with $1\leq z\leq m-1$, then $k= p + z'\sqrt{p}+n$ with $1\leq z'\leq m-1$. Thus 
the product $a_k a_{\ell}=0$, by assumption. 
In conclusion, the only left summand in the coefficient \eqref{coeff2p+msqrtp+n-2} is 
\[
\gamma_{(p,p+m\sqrt{p}+n)}a_{p} a_{p+m\sqrt{p}+n} = 0.
\]
Since $a_p\neq 0$ and $\gamma_{(p,p+m\sqrt{p}+n)}\neq 0$ (because $m\sqrt{p}+n\neq \sqrt{p}$ for all choices of $m,n$ defined above), we find $a_{p+m\sqrt{p}+n} = a_{p+t}=0$. 
\end{proof}

Note that we have already proven that $a_{p+t'}=0$ for all $1\leq t'\leq \sqrt{p}-1$. This is the case $m=1$ in
the assumption of {\bf Claim 1}. Then applying the conclusion of {\bf Claim 1} iteratively, we find that 
all the coefficients of $f$ vanish, {\it unless} they are the following $(\sqrt{p}+1)$ coefficients: $a_p, a_{p+\sqrt{p}}, a_{p+2\sqrt{p}},\ldots, a_{2p}$. To conclude, we have to show that, except $a_p$, they all must vanish. 

To this aim, for $0\leq s,r \leq \sqrt{p}$ and $(s,r)\neq (0,0)$, we look at the coefficient of $a_{p+s\sqrt{p}}a_{p+r\sqrt{p}}$ 
in the coefficient of the monomial $x_1^{2p+(s+r)\sqrt{p}-2}x_2^{2p-(s+r)\sqrt{p}-2}$ in $H(f)$. This 
has the form
\[
\gamma_{(p+s\sqrt{p}, p+r\sqrt{p})} = 2p(2p-1)(ps^2-2srp + pr^2 + sr-p). 
\]
We study the vanishing (and the sign) of $F(p,s,r) = ps^2-2srp + pr^2 + sr-p$. Write
\[
F(p,s,r) = p(s-r)^2 + sr -p. 
\]
If $s=r$, then $s=r < \sqrt{p}$, because the exponent of $x_1$ must satisfy $2p+(s+r)\sqrt{p}-2\leq 4p-4=2d-4$. 
Thus, if $s=r$ then $F(p,s,s)<0$. 

If $s,r \geq 1$ and $s\neq r$, then $p(s-r)^2-p\geq 0$ and the product $sr\geq 1$. So $F(p,s,r)>0$. 
If $s=0$ (or $r=0$), then $F(p,0,r) = pr^2 - p = 0$ if and only $r=1$ (or $s=1$). 
Otherwise, if $s=0$, then $F(p,0,r)>0$ for any $r\geq 2$. In conclusion, $\gamma_{p+s\sqrt{p}, p+r\sqrt{p}}\neq 0$ unless $s=0$ and $r=1$ (or the way around). This is 
exactly the case of \eqref{coefft^2}. \\

{\bf Claim 2:}
Suppose that for all $0\leq i\leq u -1<\sqrt{p}-1$ we know that $a_{p+(\sqrt{p}-i)\sqrt{p}}=a_{2p-i\sqrt{p}}=0$.  
Then $a_{p+(\sqrt{p}-u)\sqrt{p}}=0$. 
\begin{proof}
The remaining coefficients are $a_{p+\sqrt{p}}, a_{p+2\sqrt{p}},\ldots, a_{2p}$ (besides $a_p$) and all the others vanish.

Note that $2p-i\sqrt{p} > p+(\sqrt{p}-u)\sqrt{p}$. In other words, we are assuming 
that all the $a_h$ with $h> p+(\sqrt{p}-u)\sqrt{p}$ vanish. 

We look at the coefficient of the monomial $x_1^{4p-2u\sqrt{p}-2}x_2^{2u\sqrt{p}-2}$ in $H(f)$,
which must vanish by assumption. This coefficient has the form 
\begin{equation}\label{coeff4p-2u}
\sum_{k+\ell = 4p-2u\sqrt{p}} \gamma_{(k,\ell)} a_{k}a_{\ell} = 0. 
\end{equation}
Write $k=p+h'\sqrt{p}$ and $\ell = p + h''\sqrt{p}$. Then 
$k+\ell = 4p-2u\sqrt{p} = 2p + (h'+h'')\sqrt{p}$. Therefore $h'+h'' = 2(\sqrt{p}-u)$. If
$h' \geq \sqrt{p}-u +1$, then $k = p + h'\sqrt{p} > p+(\sqrt{p}-u)\sqrt{p}$ and hence $a_{k}=0$, by assumption. 

Thus the coefficient \eqref{coeff4p-2u} becomes
\[
\gamma_{(p+(\sqrt{p}-u)\sqrt{p}, p+(\sqrt{p}-u)\sqrt{p})} a^2_{p+(\sqrt{p}-u)\sqrt{p}}=0.
\]
Since $\gamma_{(p+(\sqrt{p}-u)\sqrt{p}, p+(\sqrt{p}-u)\sqrt{p})}\neq 0$ by the part before this claim, 
we find $a_{p+(\sqrt{p}-u)\sqrt{p}}=0$. 
\end{proof}

Now, look at the coefficient of the monomial $x_1^{4p-\sqrt{p}-2}x_2^{\sqrt{p}-2}$. The only coefficient
appearing here is the product $\gamma_{(2p, p+(\sqrt{p}-1)\sqrt{p})}a_{2p} a_{p+(\sqrt{p}-1)\sqrt{p}}$, which must vanish.
Since $\gamma_{(2p, p+(\sqrt{p}-1)\sqrt{p})}\neq 0$ by the part of the proof before {\bf Claim 2}, either $a_{2p}$ or $a_{p+(\sqrt{p}-1)\sqrt{p}}$ vanishes. 

If $a_{2p} = 0$, applying {\bf Claim 2}, we are done. If $a_{p+(\sqrt{p}-1)\sqrt{p}}=0$, then applying
{\bf Claim 2}, all coefficients $a_{p+\sqrt{p}}, a_{p+2\sqrt{p}},\ldots, a_{2p}$ are zero, except
possibly $a_{2p}$. 

However, assuming $a_{2p}\neq 0$ (with all the other coefficients being zero, except $a_p$) leads 
to an immediate contradiction with the assumption on $H(f)$. In conclusion, $f=a_{d/2}x_1^{d/2}x_2^{d/2}$ and this establishes the statement.
\end{proof}
\end{theorem}

\subsection{Stabilizers and Jacobian rings}
\begin{lemma}\label{autofhess}
One has the following descriptions for the stabilizers: 
\begin{enumerate}
\item[(i)] $\mathrm{Stab}_{\mathrm{GL}(\mathbb C^n)}((x_1\cdots x_n)^{d-2}) \cong (\mathbb C^{*})^{n-1}\rtimes \mathfrak S_n$. 
\item[(ii)]  Let $h = (x_1\cdots x_n)^{d-2}$ and let $\mathcal V(h)$ be the corresponding projective hypersurface in the projective space $\mathbb P^{n-1}_{\CC}$ equipped with the natural action of $\mathrm{PGL}(\mathbb C^n)$. Then $\mathrm{Stab}_{\mathrm{PGL}(\mathbb C^n)}(\mathcal V(h)) \cong (\mathbb C^{*})^{n}\rtimes \mathfrak S_n/\mathbb C^{*}$. 
\end{enumerate}
\begin{proof}
(i). Let $A\in \mathrm{Stab}_{\mathrm{GL}(\mathbb C^n)}((x_1\cdots x_n)^{d-2})$ be an element 
of the stabilizer in $\mathrm{GL}(\mathbb C^n)$ of the indicated polynomial. Since a polynomial ring is a unique factorization domain and the $x_i$'s are irreducible in this ring, for each $i$ we find $A\circ x_i = \gamma_i x_{\sigma(i)}$, for some $\sigma\in \mathfrak S_n$ and $\gamma_i\in \mathbb C^{*}$ with $\prod_{i=1}^n \gamma_i^{d-2} = 1$. The last condition provides an isomorphism with the $(n-1)$-dimensional algebraic torus. 
Statement (ii) is proven similarly (here we have to further quotient out by the $\mathbb C^*$ given by the scalar multiples of the identity). 
\end{proof}
\end{lemma}

\begin{lemma}\label{inclusionStab+lowerbound}
Let $f\in \CC[x_1,\dots ,x_n]_d$ and $h = H(f)\in \CC[x_1,\ldots, x_n]_{n(d-2)}$ be its Hessian.
\begin{enumerate}
\item[(i)] Their stabilizers in $\mathrm{SL}(\CC^n)$ satisfy $\mathrm{Stab}_{\mathrm{SL}(\CC^n)}(f)\subset \mathrm{Stab}_{\mathrm{SL}(\CC^n)}(h)$. 
\item[(ii)] Suppose $h= \alpha (x_1\cdots x_n)^{d-2}\neq 0$ and let $s=\dim \mathrm{Stab}_{\mathrm{SL}(\CC^n)}(f)$. Then every irreducible component of the fiber $H^{-1}(h)$ containing $f$ has dimension $\geq n-1-s$. 
\end{enumerate}
\begin{proof}
\noindent (i). Since the Hessian map is an $\mathrm{SL}(\CC^n)$-coinvariant \cite[\S 5.1]{Dolg}, for any $A\in \mathrm{Stab}_{\mathrm{SL}(\CC^n)}(f)$, we have $h = H(A\circ f) = A\circ h$, which proves the containment. \\
\noindent (ii). Let $X\subset H^{-1}(h)$ be an irreducible component containing $f$ and let $G$ be the connected component of the identity in $\mathrm{Stab}_{\mathrm{SL}(\CC^n)}(h)$, which has dimension $\dim(G) =n-1$. Note that the orbit $G\cdot f\subset X$ and so $\dim X\geq \dim(G\cdot f)\geq n-1-s$. 
\end{proof}
\end{lemma}

The two statements in Lemma \ref{inclusionStab+lowerbound} may be used as necessary conditions for the equality $H(f) = \alpha (x_1\cdots x_n)^{d-2}\neq 0$ to be satisfied by a given $f\in \CC[x_1,\ldots, x_n]_d$. An application of that comes next.

\begin{proposition}\label{rank=n+1}
Let $n\geq 2$ and $d\geq 3$. There is no $g\in \CC[x_1,\dots ,x_n]_d$ such that $H(g) = \alpha (x_1\cdots x_n)^{d-2}\neq 0$ with complex Waring rank $\mathrm{rk}_{\CC}(g) = n+1$. 
\begin{proof}
Up to the action of $\mathrm{GL}(\CC^n)$, we may assume that 
$g = x_1^d+\cdots +x_n^d+\ell^d$, where $\ell = \alpha_1x_1+\cdots + \alpha_n x_n$, for some $\alpha_i\in \CC$. If $\ell\in \langle x_{i_1},\ldots, x_{i_s}\rangle$ for some $s\leq n-1$, then setting $h = g-\sum _{j\neq i_1,\ldots, i_s} x_j^d
\in \CC[x_{i_1},\ldots, x_{i_s}]$, we find that $H(g) = \beta H(h)\cdot \prod_{j\neq i_1,\ldots, i_s} x_j^{d-2}$. 
Thus we reduce to the case where $\alpha_i\neq 0$ for all $i$. 

It is a direct computation to see that $\dim \mathrm{Stab}_{\mathrm{GL}(\CC^n)}(g)=0$. By the assumption $H(g) = \alpha (x_1\cdots x_n)^{d-2}\neq 0$ and Lemma \ref{inclusionStab+lowerbound}, the dimension of the kernel of the differential at $g$ satisfies $\dim \mathrm{Ker}(dH_g)\geq n-1$. (This is because if  the dimension of the fiber at $H(g)$ satisfies $\dim H^{-1}(H(g)) \geq q$ then $\dim \mathrm{Ker}(dH_g)\geq q$, for any $q\in \NN$.) Now, following the argument in the proof of \cite[Lemma 7.2]{CO}, we find that $\mathrm{Ker}(dH_g)=0$, which is a contradiction. 
\end{proof}
\end{proposition}

One of the tools at our disposal is a finite-dimensional algebra attached to a smooth polynomial called the {\it Jacobian ring}. Although in Proposition \ref{propmons} we will establish a sufficient condition on a smooth homogeneous polynomial with a monomial Hessian to be a Fermat polynomial, this approach seems to be weak towards answering Question \ref{quest} in the smooth case. 

\begin{definition}
Let $g\in \CC[x_1,\ldots,x_n]_d$. The Jacobian ring of $g$ is the quotient ring $R(g) = \mathbb C[x_1,\ldots, x_n]/J(g)$, where $J(g) = (\partial g/\partial x_1, \ldots, \partial g/\partial x_n)$. If $g$ is a smooth homogeneous polynomial, then $R(g)$ is a zero-dimensional local ring and a finite-dimensional graded Gorenstein $\CC$-algebra \cite[Corollary 4.2]{Huy}. Moreover, the highest non-zero degree summand (the {\it socle}) of $R(g)$ is $R(g)_{n(d-2)}\cong \mathbb C$ and generated by the class of its Hessian, i.e., $H(g)\notin J(g)$. 
\end{definition}

\begin{proposition}\label{propfermat}
Let $f$ be a Fermat polynomial. Suppose a smooth homogenous polynomial $g\in \mathbb C[x_1,\ldots, x_n]$ of degree $d\geq 3$ in $n\geq 2$ variables is such that its Jacobian ring satisfies $R(g) = R(f)$ and its Hessian $H(g) = \alpha (x_1\cdots x_n)^{d-2}$. Then $g$ is a Fermat polynomial. 
\begin{proof}
Since $R(g) = R(f)$, a result of Donagi \cite[Proposition 4.9]{Huy} shows that there exists $A\in \mathrm{PGL}(\mathbb C^n)$ such that $A\circ \mathcal V(f) = \mathcal V(g)$, where these are the corresponding degree $d$ smooth projective hypersurfaces. Note that $A$ must fix $\mathcal V(H(f))$, as $H(g)$ is a multiple of $H(f)$. Therefore $A\in \mathrm{Stab}_{\mathrm{PGL}(\mathbb C^n)}(\mathcal V(H(f)))$. By Lemma \ref{autofhess}(ii), one then finds $g = \beta_1x_1^d + \cdots + \beta_n x_n^d$. 
\end{proof}
\end{proposition}

\begin{proposition}\label{propmons}
Suppose a smooth homogenous polynomial $g\in S = \mathbb C[x_1,\ldots, x_n]$ of degree $d\geq 3$ in $n\geq 2$ variables is such that $H(g) = x_1^{b_1}\cdots x_n^{b_n}$. Suppose its Jacobian ideal $J(g)$ is a monomial ideal. Then $H(g) = (x_1\cdots x_n)^{d-2}$ and $g$ is a Fermat polynomial. 
\begin{proof}
The dimension of the Jacobian ring $R(g)$ as a complex vector space depends only on $n$ and $d$; see \cite[Proposition 4.3]{Huy}. Hence it coincides with the one of $R(f)$, where $f$ is a Fermat polynomial. One then has $\dim_{\mathbb C} R(g) = \dim_{\mathbb C} R(f) = (d-1)^n$. 

To see the last equality, let $f = \alpha_1x_1^d + \cdots + \alpha_n x_n^d$ be a Fermat polynomial. Its Jacobian ring is the quotient $R(f) = S/(x_1^{d-1},\ldots, x_n^{d-1})$. A monomial  $\CC$-basis for this algebra is formed by all monomial divisors of $H(f) = \alpha (x_1\cdots x_n)^{d-2}$, whose cardinality is $(d-1)^n$. 

Since $R(g)$ is a finite-dimensional graded artinian Gorenstein $\CC$-algebra, Macaulay's theorem \cite[Lemma 2.12]{IK99} gives a bijection between these Gorenstein algebras, whose socle is in degree $n(d-2)$, and homogenous polynomials of degree $n(d-2)$ up to $\CC^{*}$. This implies that there exists $h\in T=\mathbb C[y_1,\ldots,y_n]_{n(d-2)}$
such that $J(g) = \mathrm{Ann}(h)$, where $S$ acts by differentiation on $T$. Recall that $H(g) = x_1^{b_1}\cdots x_n^{b_n}\notin J(g) = \mathrm{Ann}(h)$. 
Since $\mathrm{Ann}(h)$ is a monomial ideal and $\dim_{\CC} R(g)_{n(d-2)}=1$, every monomial of degree $n(d-2)$ different from $H(g)$ is in $\mathrm{Ann}(h)$. Thus, writing a monomial expansion of $h$, we find that $h=y_1^{b_1}\cdots y_n^{b_n}$, up to scaling. So $\mathrm{Ann}(h) = (x_1^{b_1+1},\ldots, x_n^{b_n+1})$ and then $(d-1)^n = \dim_{\mathbb C} R(g) = \dim_{\mathbb C} T/\mathrm{Ann}(h) = \prod_{i=1}^n (b_i+1)$. 
Moreover, by definition of the Hessian, we have $\sum_{i=1}^n (b_i+1) = \left(\sum_{i=1}^n b_i\right)+n= n(d-2)+n = n(d-1)$. 

For any finite collection of nonnegative real numbers $\lbrace b_i+1\rbrace_{i\in [n]}$, we have the inequality between 
their arithmetic and geometric means (the AM-GM inequality): 
\begin{equation}\label{amgm}
d-1 = \frac{\sum_{i=1}^n(b_i+1)}{n} \geq \sqrt[n]{\prod_{i=1}^n (b_i+1)} = d-1,
\end{equation}
where the first and last equalities are consequences of the two identities on the $b_i$'s. Equality in the AM-GM inequality \eqref{amgm} is verified if and only if $b_1+1 = \cdots = b_n+1$. This is the case, and so $b_i = d-2$ for all $1\leq i\leq n$. Therefore we find $H(g) = (x_1\cdots x_n)^{d-2}$ and $J(g) = \mathrm{Ann}(h) = (x_1^{d-1},\ldots, x_n^{d-1}) = J(f)$. By Proposition \ref{propfermat}, $g$ is a Fermat polynomial. 
\end{proof}
\end{proposition}

We do not know whether after removing the assumption on $H(g)$ from Proposition \ref{propmons} the conclusion still holds true. Note that the homogeneous polynomials in Remark \ref{exampleothermons} are all singular. 
 
\begin{conjecture}\label{onlyfermat}
Suppose a homogeneous polynomial $g\in S = \mathbb C[x_1,\ldots, x_n]$ of degree $d\geq 3$ in $n\geq 2$ variables is such that its Hessian $H(g)$ is a monomial. Then $g$ is smooth if and only if $H(g) = (x_1\cdots x_n)^{d-2}$ and $g$ is a Fermat polynomial. 
\end{conjecture}

\begin{definition}\label{Shepgroups}
Let $A\in \mathrm{GL}(\mathbb C^n)$ and let $I_n$ be the $n\times n$ identity matrix. The linear transformation $A$ is a {\it unitary reflection} if it has finite order and $\mathrm{Ker}(A-I_n)$ is a codimension-one subspace, i.e. $A$ has finite order and fixes pointwise a hyperplane in $\mathbb C^n$. 
A finite group generated by unitary reflections is called a {\it unitary group generated by reflections} ({\it u.g.g.r.}). \end{definition}

\begin{remark}
A complete classification of u.g.g.r. was found by Shephard and Todd \cite{ST}.  These include as particular cases the finite Euclidean reflection groups (called {\it Coxeter groups}).  An interesting subfamily of u.g.g.r. is the one of {\it Shephard groups} arising as a symmetry group of a {\it regular complex polytope}, defined and classified by Shephard \cite{Shephard}. 

The equivalence stated in \cite[Theorem 5.10]{OS} characterises Shephard and Coxeter groups among all the u.g.g.r. groups. Item (v) in {\it loc. cit.} states that the Hessian of a minimal degree invariant form under a Shephard group $G$ is a suitable product of powers of linear functionals (each functional corresponds to a hyperplane fixed by an element in $G$).
This very last statement and its similarity with Question \ref{quest} was our motivation to look at these finite groups.  
\end{remark}

\begin{example}\label{g1n}
The group $G(d,1,n) = \mathbb Z_d^n \rtimes \mathfrak S_n \subset (\mathbb C^*)^{n}\rtimes \mathfrak S_n$ is a 
Shephard group.  As a matrix group, it may be realized as the group of $n\times n$ permutation matrices 
whose nonzero entries are $d$-th roots of unity. 
\end{example}

\begin{example}\label{gdnn}
Let $n,d\geq 2$ and $n|d$. As a matrix group, the subgroup $G(d, n, n)\subset G(d,1,n)$ consists of permutation matrices of the form $\mathrm{Diag}(\theta^{a_1},\ldots, \theta^{a_n})\circ \sigma$, where $\sigma\in \mathfrak S_n$ acts by permutation and $\mathrm{Diag}$ is a diagonal matrix with the indicated entries,  $\theta$ is a primitive $d$-th root of unity and $\sum_{i=1}^n a_i\equiv 0 \mod n$. The group $G(d, n, n)$ is a u.g.g.r. but not a Shephard group. In particular, \cite[Theorem 5.10]{OS}(v) fails. 
\end{example}

\begin{remark}\label{stabsofhompoly}
The groups described above are contained in the stabilizers of homogeneous polynomials that are related to Question \ref{quest}. 
\begin{enumerate}
\item[(i)] Let $f=x_1^d+\cdots +x_n^d$. Then $\mathrm{Stab}_{\mathrm{GL}(\mathbb C^n)}(f) = G(d,1,n)$. 
\item[(ii)] Let $n,d\geq 2$ and $n|d$. If $f= (x_1\cdots x_n)^{d/n}$ then $G(d,n,n) \subset \mathrm{Stab}_{\mathrm{GL}(\mathbb C^n)}(f)$. 
\end{enumerate}
\end{remark}

\begin{proposition}\label{gdnnprop}
Let $n\geq 2$ and $d\geq 3$ with $n|d$. Let $f\in \CC[x_1,\ldots,x_n]_d$ be such that its Hessian $H(f) = \alpha (x_1\cdots x_n)^{d-2}\neq 0$. Then $f=  \beta (x_1\cdots x_n)^{d/n}$ or $f=\beta(x_1^d + \cdots + x_n^d)$ if and only if $G(d,n,n)\subset \mathrm{Stab}_{\mathrm{GL}(\mathbb C^n)}(f)$. 
\begin{proof}
One implication is Remark \ref{stabsofhompoly}. Suppose $G(d,n,n)\subset \mathrm{Stab}_{\mathrm{GL}(\mathbb C^n)}(f)$,
i.e. $f$ is an invariant under $G(d,n,n)$.  Then $f = \alpha_1(x_1^d + \cdots + x_n^d) + \alpha_2 (x_1\cdots x_n)^{d/n}$ as the invariant ring $\CC[x_1,\ldots,x_n]^{G(d,n,n)}$ is the polynomial ring \cite[Theorem 3, Appendix III.5]{Serre} generated by the elementary symmetric functions of degrees $1,\ldots, n-1$ evaluated at $x_1^d, \ldots, x_n^d$ and by $(x_1\cdots x_n)^{d/n}$ \cite[Part II, \S 6]{ST}. 

Suppose $\alpha_1,\alpha_2\neq 0$. In the monomial expansion of $H(f)$ we look for monomials divisible by $(x_1\cdots x_{n-2})^{d-2+2\frac{d}{n}}$. Equivalently, we search for monomials $m$ in the expansion of $H(f)$ such that the $x_i$-degree of $m$, for $i\in \lbrace 1,\cdots, n-2\rbrace$, is at least $d-2+2\frac{d}{n}$. One may write
\[
\deg_{x_i}(m) = h_1(d-2) + h_2\left(\frac{d}{n}-2\right) + h_3\left(\frac{d}{n} -1\right) + h_4\frac{d}{n} \geq d-2+2\frac{d}{n}. 
\]
Here $h_1$ and $h_2$ count a choice of a monomial from the $(i,i)$-th entry of the Hessian matrix $\mathrm{Hess}(f) = [\partial^2 f/\partial x_i \partial x_j]_{ij}$, $h_3$ counts a choice of a monomial from the $i$th row or from the $i$th column (but {\it not} from the $(i,i)$-th entry), and $h_4$ counts a choice of a monomial neither from the $i$th row nor from the $i$th column. 
Thus $0\leq h_1+h_2\leq 1, 1\leq h_1+h_2+h_3\leq 2$, and $h_1+h_2+h_3+h_4\leq n$. 

Suppose $h_1=0$. If $h_2=0$, then $h_3=2$ and $h_4\leq n-2$. So $h_3\left(\frac{d}{n} -1\right) + h_4\frac{d}{n}\leq 2(\frac{d}{n}-1) + (n-2)\frac{d}{n} = d-2 < d-2+2\frac{d}{n}$ and hence this case is not possible. Similarly, if $h_2 = 1$, then $h_3=0$ and this case is not possible. 
Therefore $h_1=1$ and $h_2=h_3=0$. Hence $h_1(d-2) + h_4\frac{d}{n} = d-2 + h_4\frac{d}{n}\geq d-2+2\frac{d}{n}$, which implies $h_4\geq 2$. Since this holds for every $i\in \lbrace 1,\ldots, n-2\rbrace$, the monomials we are looking for can only be monomials 
in the expansion of the product between $\gamma (x_1\cdots x_{n-2})^{d-2}$ ($\gamma\neq 0$ because $\alpha_1\neq 0$) and the 
determinant of the lower-right $2\times 2$ minor of $\mathrm{Hess}(f)$ given by
\begin{tiny}
\begin{equation}\label{product}
\begin{vmatrix} 
\alpha_1d(d-1)x_{n-1}^{d-2} + \alpha_2 \frac{d}{n}\left(\frac{d}{n}-1\right) (x_1\ldots x_{n-2})^{\frac{d}{n}}x_{n-1}^{\frac{d}{n}-2}x_n^{\frac{d}{n}} &&  \alpha_2 \left(\frac{d}{n}\right)^2 (x_1\ldots x_{n-2})^{\frac{d}{n}}x_{n-1}^{\frac{d}{n}-1}x_{n}^{\frac{d}{n}-1} \\

\alpha_2 \left(\frac{d}{n}\right)^2 (x_1\ldots x_{n-2})^{\frac{d}{n}}x_{n-1}^{\frac{d}{n}-1}x_{n}^{\frac{d}{n}-1}  &&  \alpha_1d(d-1)x_{n}^{d-2} + \alpha_2 \frac{d}{n}\left(\frac{d}{n}-1\right) (x_1\ldots x_{n-2})^{\frac{d}{n}}x_{n-1}^{\frac{d}{n}}x_n^{\frac{d}{n}-2} \\
\end{vmatrix}.
\end{equation}
\end{tiny} 

From \eqref{product}, we see that the monomial $(x_1\cdots x_{n-2})^{d-2+2\frac{d}{n}}x_{n-1}^{2\frac{d}{n}-2}x_n^{2\frac{d}{n}-2}$ appearing in $H(f)$
has coefficient 
\[
\gamma \alpha_2^2\cdot \left[\left(\frac{d}{n}\right)^2\left(\frac{d}{n}-1\right)^2 - \left(\frac{d}{n}\right)^4\right] = \gamma \alpha_2^2\left(\frac{d}{n}\right)^2\left(1-2\frac{d}{n}\right)\neq 0. 
\]
This is a contradiction and hence either $\alpha_1= 0$ or $\alpha_2= 0$. 
\end{proof}
\end{proposition}

\begin{conjecture}\label{conjshep}
Let $n\geq 2$ and $d\geq 3$. Let $f\in \mathbb C[x_1,\ldots, x_n]_d$ be such that
its Hessian is $H(f) = \alpha (x_1\cdots x_n)^{d-2}\neq 0$. 
Then, up to conjugation, $\mathrm{Stab}_{\mathrm{GL}(\CC^n)}(f)$ contains a product of groups
of the form $G(d, 1, n-q)\times G(d, q, q)$ \textnormal{(}with $G(d,1,0)$ and $G(d,0,0)$ being trivial groups by convention\textnormal{)}, and $f$ is an invariant of minimal degree for such a group. 
\end{conjecture}

\section{Local Koiran-Skomra's question}

We formulate a local version of the Koiran-Skomra's Question \ref{quest}: whether there exist
homogeneous polynomial solutions $g$ to the equality $H(g) = \alpha (x_1\cdots x_n)^{d-2}\neq 0$, that are close to a Fermat polynomial $f = \alpha_1x_1^d + \cdots + \alpha_n x_n^d$. To deal with this local version, we first compute the differential of the Hessian map. 

\begin{lemma}[{\bf \cite[Lemma 7.2]{CO}}] \label{formula}
Let $f, g\in \CC[x_1,\ldots, x_n]_d$ be two homogeneous polynomials of degree $d$. The linear function
\[
dH_f: \CC[x_1,\ldots, x_n]_d\rightarrow \CC[x_1,\ldots, x_n]_{n(d-2)},
\] 
defined as 
\[
g\mapsto \frac{d}{dt} H(f+tg)_{|t=0}
\] 
is the differential at $f$ of the Hessian map $H: \CC[x_1,\ldots, x_n]_d\rightarrow \CC[x_1,\ldots, x_n]_{n(d-2)}$. The image $dH_f(g)$ is the sum of the determinants of $n$ matrices $H^i$ whose $i$-th row is $H^i_{ij}= \frac{\partial^2 g}{\partial x_i \partial x_j}$, and whose $k$-th row for $k\neq i$ is $H^i_{kj} =\frac{\partial^2 f}{\partial x_k\partial x_j}$.  
\begin{proof}
Viewing  $p(t) = H(f+tg) = \det(\mathrm{Hess}(f+tg))$ as a polynomial in $t$, the image of $g$ under $dH_f$ is by definition the coefficient of $t$ in $p(t)$. 
By definition of determinant, every contribution to this coefficient must involve a unique row of the matrix $\mathrm{Hess}(f+tg)$: the $i$-th such contribution is the determinant of $H^i$ in the statement. 
\end{proof}
\end{lemma}

\begin{proposition}\label{prop:dim of ker fermat}
Let $f = \sum_{i=1}^n x_i^d \in \CC[x_1,\ldots, x_n]_d$. 
\begin{enumerate}
\item[(i)] If $d\geq n+1$, then $\dim(\mathrm{Ker}(dH_f))=n-1$. 
\item[(ii)] If $d\leq n$, then $\dim(\mathrm{Ker}(dH_f))=n-1+\binom{n}{d}$. 
\end{enumerate}
\begin{proof}
\noindent (i). Let $g \in \CC[x_1,\ldots, x_n]_d$. Then, using the formula in Lemma \ref{formula}, 
we find 
\[
dH_f(g) = (d(d-1))^{n-1}\cdot \sum_{i=1}^n \prod_{j\neq i} x_j^{d-2}\frac{\partial^2 g}{\partial x_i^2}. 
\]
Denote $g_i = \prod_{j\neq i} x_j^{d-2}\frac{\partial^2 g}{\partial x_i^2}$. For two distinct indices $0\leq h, k\leq n$,
consider the monomial expansion of $g_h$ and $g_k$. Note that $\deg(g_h)=\deg(g_k) = n(d-2)$.  
A monomial shared by $g_h$ and $g_k$ must be divisible by both $\prod_{j\neq h} x_j^{d-2}$ and $\prod_{j\neq k} x_j^{d-2}$. Thus such a monomial is divisible by $\prod_{i=1}^{n} x_i^{d-2}$. Since any monomial of the $g_i$ has degree $n(d-2)$, the only shared monomial by any two 
of the $g_i$'s is $\prod_{i=1}^{n} x_i^{d-2}$. 

Suppose $g\in \mathrm{Ker}(dH_f)$. Then the condition $dH_f(g)=0$ implies that the coefficients of all monomials appearing only in a single $g_i$ 
vanish. Since $d\geq n+1$, all the coefficients of $g$ appear in a unique $g_i$, except the coefficients $\beta_i$ of the powers $x_i^d$ appearing in the monomial expansion of $g$.
The fact that the only shared monomial by all the $g_i$'s is $\prod_{i=1}^{n} x_i^{d-2}$ implies the linear equation $\sum_{i=1}^n \beta_i=0$. (Note that this is the tangent space of the algebraic torus in the fiber at $f$.) Thus $\dim(\mathrm{Ker}(dH_f)) = n-1$. \\ 
\noindent (ii). Keep the notation from above. Since $d\leq n$, the coefficients of monomials in $g$ annihilated by all the quadratic differentials, i.e. the square-free monomials $x_{i_1}\cdots x_{i_d}$ for some $1\leq i_1,\ldots, i_d\leq n$, do not appear 
in any of the $g_i$'s summands. There are $\binom{n}{d}$ of those. This shows the lower bound $\dim(\mathrm{Ker}(dH_f))\geq n-1+\binom{n}{d}$. The equality comes from the direct verification that any other coefficient in $g$ appears as coefficient of a unique monomial in a unique $g_i$ and thus must vanish. 
\end{proof}
\end{proposition}

\begin{remark}
Proposition \ref{prop:dim of ker fermat} implies that, when $d\geq n+1$, the local version of Koiran-Skomra's Question \ref{quest} has a positive answer. \end{remark}

For binary homogeneous polynomials, we offer a description of the differential at all monomials. We believe that the differential is important to understand better the Hessian map. For this part we were inspired by \cite{CO}, where the authors provide ample evidence on how useful the differential is. 

\begin{proposition}\label{binarymonomials}
Let $d\geq 2$ and let $f = x_1^{k}x_2^{d-k}\in \CC[x_1,x_2]_d$, with $0 < k\leq d$. 
\begin{enumerate}
\item[(i)] If $k=d$, then $\dim(\mathrm{Ker}(dH_f)) = 2$.

\item[(ii)] If $k=d-1$, then $\dim(\mathrm{Ker}(dH_f)) = 1$, unless $d=2$. In the last case, $\dim(\mathrm{Ker}(dH_f)) = 2$. 

\item[(iii)] If $k\leq d-2$, then $\dim(\mathrm{Ker}(dH_f))$ is the number of indices $0\leq j\leq d$ such that 
\[
d(k-j)^2 - d(k+j) + 2kj=0.
\] 
 \item[(iv)] Let $d\geq 3$ and $k\leq d-1$. Then $\dim(\mathrm{Ker}(dH_f)) \le 1$, unless $d=2k$ and
$k$ is a square. In this case $\mathrm{Ker}(dH_f)$ is the span of $x_1^{k+\sqrt{k}}x_2^{k-\sqrt{k}}$ and $x_1^{k-\sqrt{k}}x_2^{k+\sqrt{k}}$.
\end{enumerate}
\begin{proof}
Statement (i) is a direct computation using the formula in Lemma \ref{formula}. \\
By Lemma \ref{formula}, for $g\in \CC[x_1,x_2]_d$, we have 
\begin{small}
\[
dH_f(g) = \begin{vmatrix}
\partial^2 g/\partial x_1^2 &  \partial^2 g/\partial x_1\partial x_2 \\
(d-k)k x_1^{k-1} x_2^{d-k-1} &  (d-k)(d-k-1) x_1^{k}x_2^{d-k-2}  \\
\end{vmatrix}+\begin{vmatrix}
k(k-1) x_1^{k-2}x_2^{d-k}  & (d-k)k x_1^{k-1}x_2^{d-k-1} \\
\partial^2 g/\partial x_1\partial x_2 &  \partial^2 g/\partial x_2^2 \\
\end{vmatrix}. 
\]
\end{small}
Let $g = \sum_{i=0}^d a_i x_1^i x_2^{d-i}$. Thus
\[
dH_f(g) = (d-k)(d-k-1)\left[\sum_{i=0}^d a_i i(i-1) x_1^{k+i-2} x_2^{2d-k-i-2}\right] + 
\]
\[
-2(d-k)k\left[\sum_{i=0}^d a_i(d-i)i x_1^{k+i-2} x_2^{2d-k-i-2}\right]+
\]
\[
+ k(k-1)\left[\sum_{i=0}^d a_i (d-i)(d-i-1) x_1^{k+i-2} x_2^{2d-k-i-2}\right]. 
\]
The coefficient of the monomial $a_j x_1^{k+j-2} x_2^{2d-k-j-2}$ is 
\begin{equation}\label{expressionofcoeff}
g(d,k,j)= (d-k)(d-k-1)j(j-1) -2(d-k)k(d-j)j + k(k-1)(d-j)(d-j-1) = 
\end{equation}
\[
= (d-1)\left[d(k-j)^2 - d(k+j) + 2kj\right].
\]
(ii). Suppose $k=d-1$. Then $g(d, d-1,j) = (d-1)(d-j)(d^2-dj-3d+2)$. Thus we have $g(d,d-1,d)=0$. 
For $j=d-s\leq d-1$, we find that $g(d, d-1,d-s) = (d-1)(d-j)(sd-3d+2)$. So $g(d, d-1,d-s)=0$
if and only if $d= -\frac{2}{s-3}$, which is only possible if $s=1$ or $s=2$. If $s=1$, then $d=1$,
which is excluded from the assumptions as it is a trivial case. 
For $s=2$, we have $d=2$ and $j=0$. In this case, in the expansion of $dH_f(g)$ the only vanishing coefficients are those of $a_0$ and $a_2$ and whence $\dim(\mathrm{Ker}(dH_f)) = 2$. 
Otherwise, when $k=d-1$, in the expansion of $dH_f(g)$ the only vanishing coefficient is the one of $a_d$ and so $\dim(\mathrm{Ker}(dH_f)) = 1$.\\
(iii). Using the approach in (ii) above, when $k\leq d-2$ the only vanishing coefficients of the monomials $a_j x_1^{k+j-2} x_2^{2d-k-j-2}$ are for those indices $j$ satisfying $g(d,k,j)=0$. Then it is immediate to see that $\dim(\mathrm{Ker}(dH_f))$ is the number of indices $j$ such that $g(d,k,j)=0$. (Clearly, the equation is quadratic in $j$ and hence $\dim(\mathrm{Ker}(dH_f))\leq 2$.) \\
(iv). By part (iii), $\dim \mathrm{Ker}(H_f) > 1$ if and only if the following polynomial in $j$
\begin{equation}\label{eqot1}
j^2 +j\left(-2k-1+\frac{2k}{d}\right) +k^2-k =0
\end{equation}
has two integral roots $0\le j_1< j_2\le d$. Assume the existence of such $j_1,j_2$. Thus
\begin{equation}\label{eqot2}
j_1j_2=k(k-1) \mbox{ and }  j_1+j_2 = 2k+1 - \frac{2k}{d}.
\end{equation}
Since $0<k<d$, \eqref{eqot2} gives $d=2k$ and $j_1+j_2 =2k$. So $j_1 =k-\sqrt{k}$ and $j_2=k+\sqrt{k}$. Assuming the existence of these integral solutions, $d=2k$ where $k$ is a square. 
\end{proof}
\end{proposition}

In some cases, one finds a large subset of binary homogeneous polynomials whose differential is injective: 

\begin{corollary}\label{corwithlim}
Let $d\geq 3$ be a prime and $2\leq k\leq d-2$. Let $W_k\subset \CC[x_1,x_2]_d$ be the set of homogenous polynomials of the form $\sum_{i\geq k}^d a_i x_1^{i} x_2^{d-i}$ such that $a_k\neq 0$. Then, for any $f\in W_k$, $\dim(\mathrm{Ker}(dH_f))=0$. 
\begin{proof}
For $\lambda\in \CC^*$ and $f\in W_k$, define $f_\lambda(x_1,x_2)= f(\lambda x_1, x_2)/\lambda^{k}$. One has $\dim(\mathrm{Ker}(dH(f_\lambda)))= \dim(\mathrm{Ker}(dH_f))$. 
Note that $g=\lim_{\lambda\rightarrow 0} f_\lambda(x_1, x_2) = a_k x_1^{k}x_2^{d-k}\neq 0$. By semicontinuity, we have $\dim(\mathrm{Ker}(dH_f))\leq \dim(\mathrm{Ker}(dH(g)))$. By Proposition \ref{binarymonomials}(iii), the latter is the number of indices $0\leq j\leq d$ satisfying the equation $d(k-j)^2 - d(k+j) + 2kj=0$.
If $j=0$, then $k=0$ or $k=1$, which are excluded. For $j\neq 0$, since $d$ is a prime dividing the nonzero integer $2kj$ and $k\leq d-2$, we must have $j=d$. For $j=d$, one finds $d(k-j)^2 - d(k+j) + 2kj=d(k-d)(k-d+1)$. So, given $2\leq k\leq d-2$, there is no $j$ satisfying the equation above. Thus $\dim(\mathrm{Ker}(dH_f))=0$ for all $f\in W_k$. 
\end{proof}
\end{corollary}

In Theorem \ref{binaryeven}, the exceptional cases appear when $d=2m^2$. These are also somewhat exceptional instances
from the perspective of the differential of the Hessian map, as witnessed by Proposition \ref{binarymonomials}(iv) and, as a consequence,
by the next result.

\begin{theorem}\label{kernelsofbin}
Let $d\geq 3$ and let $f\in \CC[x_1,x_2]_d$ with $H(f)\neq 0$. Then $\dim H^{-1}(H(f)) \le 1$ and $\dim (\mathrm{Ker}(dH_f))\le 1$, unless $d=2m^2$ and $f=\ell_1^{m^2}\ell_2^{m^2}$ (where $\ell_i$ are linearly independent linear forms). In the latter case, $\dim (\mathrm{Ker}(dH_f))=2$. 

\end{theorem}
\begin{proof}
First, suppose $d\neq 2m^2$. It is sufficient to prove that $\dim \mathrm{Ker}(dH_f)\le 1$. 
(This is because if  the dimension of the fiber at $H(f)$ satisfies $\dim H^{-1}(H(f)) \geq q$ then $\dim (\mathrm{Ker}(dH_f))\geq q$, for any $q\in \NN$.)
For any $k\ge 1$, let $W_k$ be the set of all $g\in \CC[x_1,x_2]_d$ such that $g=\sum _{i	\geq k}^{d} a_ix_1^ix_2^{d-i}$ and $a_k\ne 0$.
Up to the action of $\mathrm{GL}(\mathbb C^2)$, we may assume $f(0,1)=0$, i.e. we may assume $f=\sum^{d}_{i\geq 1} a_ix_1^{i}x_2^{d-i}$. Let $i_{\mathrm{min}}$ be the smallest positive integer such that $a_{i_{\mathrm{min}}} \ne 0$ and hence $f\in W_{i_{\mathrm{min}}}$. We have $i_{\mathrm{min}}<d$, because $H(f)\ne 0$ and so $1\le i_{\mathrm{min}}\le d-1$. As in the proof of Corollary \ref{corwithlim}, for any $\lambda\in \CC^*$, define $g_\lambda(x_1,x_2) = f(\lambda x_1,x_2)/\lambda^{i_{\mathrm{min}}}$. Note that $g = a_{i_{\mathrm{min}}}x_1^{i_{\mathrm{min}}}x_2^{d-i_{\mathrm{min}}} = \lim_{\lambda\rightarrow 0} g_\lambda$. One then has $\dim (\mathrm{Ker}(dH_{g_\lambda})) = \dim (\mathrm{Ker}(dH_f))$ for all $\lambda\in \CC^{*}$. By semicontinuity, we have $\dim (\mathrm{Ker}(dH_f))\le \dim (\mathrm{Ker}(dH_g))\le 1$, where the last inequality is Proposition \ref{binarymonomials}(iv). \\
Now, assume $d=2m^2$. Since $H(f)\neq 0$ and $d\geq 3$, we may assume that $f$ has at least two distinct zeros. Up to the action of $\mathrm{GL}(\CC^2)$, we may assume $f(0,1)=f(1,0)=0$, i.e. we may assume $a_0=a_d=0$. The first part of this proof works if $i_{\mathrm{min}} \ne m^2$. Thus, we may assume $i_{\mathrm{min}} = m^2$. Let $j_{\mathrm{min}}$ be the minimal positive integer such that $a_{d-j_{\mathrm{min}}}\ne 0$. Since $a_d=0$, $j_{\mathrm{min}}>0$. Define $h_\lambda(x_1,x_2) = f(x_1,\lambda x_2)/\lambda^{j_{\mathrm{min}}}$. Since $h = a_{d-j_{\mathrm{min}}}x_1^{d-j_{\mathrm{min}}}x_2^{j_{\mathrm{min}}} =\lim _{\lambda\rightarrow 0} h_\lambda$, by semicontinuity we obtain $\dim (\mathrm{Ker}(dH_f))\le \dim (\mathrm{Ker}(dH_h))$. If $j_{\mathrm{min}}<m^2$, then Proposition \ref{binarymonomials}(iv) gives $\dim (\mathrm{Ker}(dH_f))\le 1$. Otherwise, $i_{\mathrm{min}} = j_{\mathrm{min}} = m^2$ and so $f = x_1^{m^2} x_2^{m^2}$. In this case, Proposition \ref{binarymonomials}(iv) shows that $\dim(\mathrm{Ker}(H_f)) = 2$. \end{proof}

\begin{remark}
Theorem \ref{kernelsofbin} also shows that the Hessian map at the monomial counterexamples is not necessarily a local embedding. 
\end{remark}

\begin{small}

\end{small}
\end{document}